\documentclass[11pt]{amsart}
\usepackage[T1]{fontenc}
\usepackage[utf8]{inputenc}
\usepackage{amsmath,amssymb,amsthm,mathtools,microtype}
\usepackage{etoolbox}
\usepackage[hidelinks]{hyperref}
\hypersetup{pdftitle={Translation Continuity and the Structure of Infinite-Dimensional Lebesgue Spaces},pdfauthor={Daniel L. Rodriguez-Vidanes and Juan Carlos Sampedro}}
\makeatletter
\patchcmd{\@setaddresses}{\/:\space}{\/\quad}{}{}
\patchcmd{\@setaddresses}{\/:\space}{\/\quad}{}{}
\patchcmd{\@setaddresses}{\/:\space}{\/\quad}{}{}
\makeatother
\usepackage[a4paper,left=2.35cm,right=2.35cm,top=2.45cm,bottom=2.5cm]{geometry}
\newcommand{\R}{\mathbb R}
\newcommand{\N}{\mathbb N}
\newcommand{\Z}{\mathbb Z}
\newcommand{\F}{\mathcal F}
\newcommand{\B}{\mathcal B}
\newcommand{\1}{\mathbf 1}
\newcommand{\cc}{\mathfrak c}
\newcommand{\vol}{\operatorname{vol}}
\newcommand{\supp}{\operatorname{supp}}
\newcommand{\dist}{\operatorname{dist}}
\newtheorem{theorem}{Theorem}[section]
\newtheorem{lemma}[theorem]{Lemma}
\theoremstyle{definition}
\newtheorem{example}[theorem]{Example}
\numberwithin{equation}{section}

\title[Translation Continuity and the Structure of Infinite-Dimensional Lebesgue Spaces]{Translation Continuity and the Structure of Infinite-Dimensional Lebesgue Spaces}

\author[D. L. Rodr\'iguez-Vidanes]{Daniel L. Rodr\'iguez-Vidanes}
\address{Grupo de Investigaci\'on de An\'alisis Matem\'atico y Aplicaciones, Departamento de Matem\'atica Aplicada a la Ingenier\'ia Industrial, Escuela T\'ecnica Superior de Ingenier\'ia y Dise\~no Industrial, Universidad Polit\'ecnica de Madrid, Ronda de Valencia 3, Madrid, 28012, Spain \newline \indent \href{https://orcid.org/0000-0002-1016-096X}{ORCID \texttt{0000-0002-1016-096X} }}
\email{dl.rodriguez.vidanes@upm.es}

\author[J. C. Sampedro]{Juan Carlos Sampedro}
\address{Departamento de Matem\'atica Aplicada y Ciencias de la Computaci\'on, Universidad de Cantabria (UC), Avenida de los Castros 46, Santander, 39005, Spain}
\email{juancarlos.sampedro@unican.es}



\thanks{The second author has been supported by the Research Grant PID2024--155890NB-I00 of the Spanish Ministry of Science and Innovation and by the Institute of Interdisciplinar Mathematics of Complutense University. Both authors were supported by the Spanish Ministry of Science, Innovation and Universities and the Spanish State Research Agency, project PID2025-171414NA-I00 (MICIU/AEI/10.13039/501100011033)}

\begin{document}
	
\begin{abstract}
For Baker's Borel-rectangle measure $\mu$ and $1\le p<\infty$, we classify the continuous vectors of each one-parameter translation group on $L^p(\mu)$. For each $v\in\R^\N$, the continuous vectors form a projection band, which coincides with the whole space when $v$ has finite support, is trivial when $v\notin\ell^1$, and is nonzero and proper otherwise. We also establish in ZFC the isometric classification considered in \cite{previous}, showing that both $L^p(\mu)$ and the kernel of restriction to integer cubes are lattice isometric to $\ell^p(\cc,L^p[0,1])$. When $v\in\ell^1\setminus c_{00}$, the continuous band and its disjoint complement have the same lattice-isometric structure, with the latter contained in the cube-restriction kernel.
\end{abstract}

\maketitle

\section{Introduction}
A problem in infinite-dimensional analysis is to determine when an integral depending on infinitely many variables can be recovered from finite-dimensional data. For probability measures on countable products, thanks to Kakutani theorem \cite{kakutani}, finite-coordinate approximations provide a natural route, that is, integrable functions can be approximated using functions of finitely many coordinates, and their integrals can be recovered from the corresponding finite-dimensional measures. 
Related $L^p$-space methods extend finite-dimensional approximation to projective limits of probability measures \cite{sampedrolp,sampedroproj}. 
For measures of infinite mass, the structure of the associated $L^p$-spaces provides a framework for decomposing integrable functions and approximating their integrals.

Baker's infinite-dimensional analogues of Lebesgue measure \cite{baker91,baker04} provide a particularly instructive example. They extend the familiar product formula for volume to suitable rectangles in $\R^\N$ while retaining countable additivity and invariance under arbitrary translations. These properties make Baker's measures natural candidates for an integration theory in infinitely many variables. The geometry of these measures reflects the infinite-dimensional setting.
For instance, the authors showed in \cite[Th.~A.7]{previous} that every nonempty open set has infinite measure. 

The purpose of this paper is to investigate the resulting $L^p$-spaces from two complementary perspectives: (1) their decomposition into standard function spaces which is relevant to integral approximation, and (2) the continuity of translations which governs the stability of functions under small shifts.

\medskip

We begin by recalling the construction and properties of Baker's measure on $\R^\N$ and fixing the notation used throughout the paper. We work with the Borel-rectangle measure introduced in \cite{baker04}. Write $X:=\R^\N$. We endow $\R$ with its usual topology and $X$ with the product topology, and denote their respective Borel $\sigma$-algebras by $\B$ and $\B_\infty$. Write $\lambda$ for the one-dimensional Lebesgue measure on $\B$. The construction starts from the family of rectangles
	$$
	\F:=\left\{\bigtimes_{n=1}^\infty A_n \colon A_n\in\B,\ \lambda(A_n)<\infty,\ \lim_{N\to\infty}\prod_{n=1}^N\lambda(A_n)\text{ exists in }[0,\infty)\right\}.
	$$
For $R=\bigtimes_{n=1}^\infty A_n\in\F$, define its volume by $\vol(R)=\lim_{N\to\infty}\prod_{n=1}^N\lambda(A_n)$. The associated outer measure is
	\begin{equation}\label{outer}
	\mu^*(E):=\inf\left\{\sum_{k=1}^\infty\vol(R_k)\colon E\subset\bigcup_{k=1}^\infty R_k,\ R_k\in\F \text{ for all } k\in \N\right\},\qquad E\subset \R^\N,
	\end{equation}
where $\inf\varnothing=\infty$, and $\mu:=\mu^*|_{\B_\infty}$. The defining properties of this measure include
	\begin{equation}\label{measure}
	\mu(R)=\vol(R)\quad(R\in\F),\qquad \mu(E+h)=\mu(E)\quad(E\in\B_\infty,\ h\in \R^\N).
	\end{equation}
For a more general treatment of infinite product measures, see \cite{sampedro}. 

Throughout, $1\leq p<\infty$. Given a measure space $(A,\Sigma,\nu)$, we denote by $L^p(A,\nu)$, or simply $L^p(\nu)$ when the measure space is clear, the real or complex Banach space of $p$-integrable measurable functions, identified modulo $\nu$-null sets and endowed with the standard $L^p$-norm. We denote by $\cc=2^{\aleph_0}$ the cardinality of the continuum. Also recall that $L^p$-spaces are Banach lattices.
We say that $Z$ is a \textit{Banach lattice} if $Z$ is a Banach space equipped with a partial order $\leq$ that makes it a vector lattice, and its norm is monotone, meaning that $|x|\le |y|$ implies $\|x\|\le \|y\|$ for $x,y\in Z$.

An elementary example illustrates why the reduction of these spaces to manageable pieces is delicate. The integer unit cubes
	$$
	C_a:=\bigtimes_{n=1}^\infty[a_n,a_n+1),\qquad a\in I:=\Z^\N,
	$$
form a partition of $\R^\N$, and each $C_a$ has $\mu$-measure one. This uncountable partition gives rise to a nontrivial kernel of simultaneous cube restriction. Indeed, the shifted cube $Q=[1/2,3/2)^\N$ satisfies $\mu(Q)=1$ and $\mu(Q\cap C_a)=0$ for every $a\in I$, since each coordinate overlap has length at most $1/2$ and its infinite product is zero. Thus $\1_Q$ is a function of integral one whose restriction to every integer cube vanishes almost everywhere.

The previous work by the authors \cite{previous} established a decomposition that isolates precisely this phenomenon. Denoting by $G^p$ the functions supported, modulo null sets, on countably many integer cubes, and setting
	$$
	B^p:=\{f\in L^p(\mu) \colon f|_{C_a}=0\text{ in }L^p(C_a,\mu)\text{ for every }a\in I\},
	$$
we have
	$$
	L^p(\mu)=G^p\oplus_p B^p,\qquad G^p\cong\ell^p(\cc,L^p[0,1]).
	$$
Here $\cong$ denotes a surjective linear isometry, and $\ell^p(J,Y)$ is the space of $p$-summable families of elements of a Banach space $Y$ indexed by $J$ endowed with the norm $\|(y_j)_{j\in J}\| = \left( \sum_{j\in J} \|y_j\|^p \right)^{1/p}$. The subspace $B^p$ is a closed kernel of simultaneous cube restriction, and the example above shows that it is nonzero. Although \cite[Th.~3.6]{previous} identified a complemented copy of $\ell^p(\cc,L^p[0,1])$ inside $B^p$, the isometric type of the entire kernel was left open. Likewise, the isometric classification of $L^p(\mu)$ required the continuum hypothesis \cite[Th.~2.4]{previous}.

\medskip

In Section~\ref{sec2}, we resolve both issues in ZFC. In fact, the method applies more generally to translation-invariant closed ideals. Recall that a \textit{closed ideal} of $L^p(\mu)$ is a closed linear subspace $\mathcal J$ such that, if $f\in\mathcal J$ and $|g|\le |f|$ almost everywhere, then $g\in\mathcal J$. For $h\in \R^\N$, the translation operator
	$$
	(T_hf)(x):=f(x-h)
	$$
is a surjective linear isometry of $L^p(\mu)$ by \eqref{measure}. We say that $\mathcal J$ is \textit{translation invariant} if $T_h\mathcal J=\mathcal J$ for every $h\in \R^\N$.

	\begin{theorem}\label{classification}
	For every $1\le p<\infty$, we have $L^p(\mu)\cong\ell^p(\cc,L^p[0,1])$ and $B^p\cong\ell^p(\cc,L^p[0,1])$.
	The same identification holds for every nonzero closed ideal of $L^p(\mu)$ invariant under all translations, and all these identifications can be chosen to be Banach lattice isometric isomorphisms.
	\end{theorem}

Theorem~\ref{classification} holds in ZFC, and in particular removes the continuum hypothesis from the classification in \cite{previous}. It also supplies a $\cc$-separable envelope in the sense of \cite[Def.~4.4]{previous}, thereby disproving \cite[Conj.~4.6]{previous}. The proof works directly with finite-measure Borel sets and maximal essentially disjoint families. We say that a family $(D_j)_{j\in J}$ of measurable sets is \textit{essentially disjoint} if $\nu(D_i\cap D_j)=0$ whenever $i,j\in J$ are distinct, and it is called \textit{maximal} when it is maximal with respect to inclusion among essentially disjoint families. Maharam's theorem \cite{maharam,fremlin} enters through the finite-measure pieces. Working at this level accommodates the failure of semifiniteness and localizability of $\mu$ established in \cite[Pr.~A.3 and Th.~A.1]{previous}.

The classification has a concrete interpretation for integration. Taking $p=1$ and choosing a lattice isometry $\Phi:L^1(\mu)\to\ell^1(\cc,L^1[0,1])$, we obtain, for every $f\in L^1(\mu)$,
	$$
	\int_X f\,d\mu=\sum_{\xi<\cc}\int_0^1(\Phi f)_\xi(s)\,ds.
	$$
Indeed, a lattice isometry on $L^1$ preserves integrals of nonnegative functions, and hence of all integrable functions. Since at most countably many components of $\Phi f$ are nonzero, the sum can be approximated by finite sums of ordinary one-dimensional integrals, with the error controlled by the remaining $\ell^1$-norm. This abstract decomposition expresses each integral through countably many one-dimensional integrals. A numerical implementation requires a concrete description of the isometry.

There is also a reduction that uses the original coordinates. As we prove in Lemma~\ref{density}, finite linear combinations of indicators of rectangles from $\F$ are dense in $L^1(\mu)$. Thus, given $f\in L^1(\mu)$, we may choose
	$$
	s_m=\sum_{k=1}^{K_m}c_{m,k}\1_{R_{m,k}}\longrightarrow f\quad\text{in }L^1(\mu),\qquad R_{m,k}=\prod_{j=1}^\infty A_{m,k,j}\in\F.
	$$
By \eqref{measure}, we have the following finite-dimensional approximation formula
\[
\int_X f\,d\mu
=\lim_{m\to\infty}\lim_{n\to\infty}
\int_{\R^n}\left(\sum_{k=1}^{K_m}c_{m,k}\prod_{j=1}^n\1_{A_{m,k,j}}(x_j)\right)d\lambda^{\otimes n}(x).
\]
For each fixed $m$, the inner limit $n\to\infty$ computes the infinite-dimensional volumes of the rectangles from their finite-dimensional volumes. The outer limit $m\to\infty$ uses the convergence $s_m\to f$ in $L^1(\mu)$ to recover the integral of $f$. This procedure applies to every function in $L^1(\mu)$, including those in $B^1$. It gives a rigorous existence principle for approximating integrals. An effective choice of approximants requires additional information about the integrand.

\medskip

Finally, in Section~\ref{sec3} we study how this structure interacts with translations. In finite dimensions, translations act strongly continuously on $L^p$ for $p<\infty$, so small shifts of any function are small in norm. In infinite dimensions, this continuity depends on both the direction and the function. Translations can move functions between the two summands of the preceding decomposition. For example, if $C=[0,1)^\N$ and $h=(1/2,1/2,\ldots)$, then $\1_C\in G^p$ and $T_h\1_C\in B^p$. Translation invariance gives $\int_X T_hf\,d\mu=\int_X f\,d\mu$ for every $f\in L^1(\mu)$. The stability of approximations under small shifts is expressed by continuity in the $L^p$-norm, which we now investigate.

To make the question precise, fix a direction $v\in X$ and consider the operators $T_{tv}$ with $t\in\R$. They satisfy $T_{sv}T_{tv}=T_{(s+t)v}$ and $T_{0v}=\mathrm{Id}$.
Thus $(T_{tv})_{t\in\R}$ is a \emph{one-parameter group of isometries}. We say that a function has a \textit{continuous translation orbit in the direction $v$} if $T_{tv}f\to f$ in $L^p(\mu)$ as $t\to0$. The functions with this property form the space
	$$
	C_v^p:=\{f\in L^p(\mu)\colon \|T_{tv}f-f\|_p\longrightarrow0\text{ as }t\to0\}.
	$$
Since the operators are isometries and satisfy the group identities, continuity at $t=0$ is equivalent to continuity at every $t\in\R$. Accordingly, the action is called \emph{strongly continuous} when $C_v^p=L^p(\mu)$. We also write
	$$
	D_v^p:=\{g\in L^p(\mu) \colon |g|\wedge|f|=0\text{ almost everywhere for every }f\in C_v^p\}
	$$
for its disjoint complement, where $\wedge$ denotes the pointwise minimum. A closed ideal is called a \emph{projection band} when it and its disjoint complement give a direct-sum decomposition of the ambient space. The following theorem shows that there are exactly three possibilities for continuous vectors. We use $c_{00}$ for the finitely supported sequences.

	\begin{theorem}\label{trichotomy}
	For every $1\le p<\infty$ and $v\in\R^\N$, we have
		$$
		\begin{cases}
		C_v^p=L^p(\mu) & \text{if } v\in c_{00},\\
		\{0\}\subsetneq C_v^p\subsetneq L^p(\mu)&\text{if }v\in\ell^1\setminus c_{00},\\
		C_v^p=\{0\}&\text{if }v\notin\ell^1.
		\end{cases}
		$$
	For every $v\in\R^\N$, the space $C_v^p$ is a projection band and $L^p(\mu)=C_v^p\oplus_p D_v^p$, with band projections commuting with every translation. If $v\in\ell^1\setminus c_{00}$, then
		$$
		C_v^p\cong D_v^p\cong B^p\cap C_v^p\cong\ell^p(\cc,L^p[0,1])
		$$
	as Banach lattices, and $B^p=(B^p\cap C_v^p)\oplus_p D_v^p$.
	\end{theorem}

Finite support characterizes the directions of strong continuity on the whole space, and summability characterizes those admitting a nonzero continuous vector. For $v\notin\ell^1$, the continuous band is zero. In the intermediate case, the continuous band and its disjoint complement have the same abstract Banach-lattice structure and distinct translation behaviour. The translation representation therefore supplies information about stability under shifts beyond the isometric classification of the underlying space.


The strong continuity of translations on infinite-dimensional function spaces depends sensitively on the underlying measure and its associated function space. Gill, Pantsulaia and Zachary \cite[Th.~2.26]{gill} studied admissible translations for a different infinite-dimensional measure construction. More directly, Sakbaev and Zavadsky \cite[Th.~1]{sakbaev} established an $\ell^1$ criterion for strong continuity of the translation group on their $L^2$-space over $\ell^\infty$. Related criteria for translation groups on Hilbert spaces associated with other invariant measures appear in \cite{busovikovSakbaev,zavadskySakbaev}. Busovikov and Sakbaev \cite[Th.~5.1]{busovikovSakbaev2022} obtained a finite-support criterion for strong continuity on the entire space in a different infinite-dimensional Hilbert-space setting.

The existence of strongly continuous invariant subspaces when the full representation is discontinuous has also been investigated. In particular, Dzhenzher and Sakbaev \cite{dzhenzherSakbaev2026} identified such a subspace for translations on an $L^2$-space associated with a finitely additive Banach measure. These results describe continuous subspaces for their respective measures and representations. For Baker's measure, Example~\ref{beyond} exhibits continuous vectors outside the closed span of the translated cubes used in \cite[Lem.~5]{sakbaev}. Our results determine both the directions of strong continuity on the entire $L^p$-space and those admitting nonzero continuous vectors, and identify the corresponding subspaces as projection bands with an explicit Banach-lattice structure.

\section{Finite-measure pieces and bands}\label{sec2}
The classification rests on the structure of the finite-measure pieces established in \cite{previous}. For every Borel set $E$ with $0<\mu(E)<\infty$, separability follows from the countable generation argument in \cite[Lem.~2.3]{previous}, while nonatomicity is proved in \cite[Pr.~A.5]{previous}. Both arguments hold in ZFC. These facts give $L^p(E,\mu)\cong L^p[0,1]$, as also recorded in \cite[Lem.~4.2]{previous}. The identification can be chosen to be a lattice isometry since the normalized measure algebra of $E$ is isomorphic to the Lebesgue probability algebra \cite[Pr.~331P]{fremlin}, with a rescaling of functions accounting for $\mu(E)$. The corresponding representation of $L^p$-spaces is discussed in \cite{lacey}.

The restriction and gluing operators in the next lemma already occur in \cite[Th.~4.3]{previous}. Maximality ensures that they account for every function in the ideal and yields a projection onto it. Each function is reconstructed from countably many pieces. A related use of maximal essentially disjoint families appears in \cite[p.~40]{gunturk}.

	\begin{lemma}\label{bands}
	Let $\mathcal J$ be a closed ideal of $L^p(Y,\Sigma,\nu)$, where $1\le p<\infty$, and let $(D_j)_{j\in J}$ be maximal among families of sets of finite positive measure with pairwise null intersections and $\1_{D_j}\in\mathcal J$ for every $j\in J$. Restriction gives a surjective lattice isometry
		$$
		\mathcal J\longrightarrow\left(\bigoplus_{j\in J}L^p(D_j,\nu)\right)_p.
		$$
	Moreover, $L^p(\nu)=\mathcal J\oplus_p\mathcal J^\perp$, where $\mathcal J^\perp$ is its disjoint complement.
	\end{lemma}

\begin{proof}
	
	Let $E$ be a measurable set with $\nu(E)<\infty$ and $\1_E\in\mathcal J$. Since the sets $D_j$ are essentially disjoint, the set $J_E:=\{j\in J:\nu(E\cap D_j)>0\}$ is countable. Moreover, the set $E':=E\setminus\bigcup_{j\in J_E}D_j$ satisfies $\1_{E'}\in\mathcal J$ by monotonicity of the ideal, and $\nu(E'\cap D_j)=0$ for every $j\in J$.
	By maximality of the family $(D_j)_{j\in J}$, we have $\nu(E')=0$. Consequently,
		\begin{equation}\label{exhaustion}
			\nu(E)=\sum_{j\in J}\nu(E\cap D_j).
		\end{equation}
	
	Now let $f\in\mathcal J$. For each $m\in\N$, the set $E_m:=\{|f|>1/m\}$ has finite measure, and $\1_{E_m}\in\mathcal J$ by monotonicity. Applying the preceding argument to $E_m$, there exists a countable subset $J_m\subseteq J$ such that $\nu\left(E_m\setminus\bigcup_{j\in J_m}D_j\right)=0$. Since $\{|f|>0\}=\bigcup_{m\in\N}E_m$, it follows that $f$ vanishes almost everywhere outside $\bigcup_{j\in J_f}D_j$, where $J_f:=\bigcup_{m\in\N}J_m$ is countable. Consequently, by essential disjointness, $\|f\|_p^p = \sum_{j\in J}\|f|_{D_j}\|_p^p$. Thus, the restriction map is a lattice isometry.
	

	To prove surjectivity, let $(g_j)_{j\in J}\in\left(\bigoplus_{j\in J}L^p(D_j,\nu)\right)_p$. For each $j\in J$, let $\widetilde g_j$ denote the extension of $g_j$ by zero outside $D_j$. Since $\1_{D_j}\in\mathcal J$, bounded truncation, monotonicity  and closedness imply that $\widetilde g_j\in\mathcal J$. Moreover, the sets $D_j$ are pairwise disjoint up to null sets, so for every finite subset $F\subseteq J$, we have $\left\|\sum_{j\in F}\widetilde g_j\right\|_p^p = \sum_{j\in F}\|g_j\|_p^p$.
	Since $(g_j)_{j\in J}$ is $p$-summable, only countably many of its coordinates are nonzero, and the series $\sum_{j\in J}\widetilde g_j$ converges in $L^p(\nu)$ to some $f\in\mathcal J$. Finally, by essential disjointness, the restriction of $f$ to each $D_j$ coincides with $g_j$ almost everywhere.
	
	For the additional part, extend the family $(D_j)_{j\in J}$ to a maximal essentially disjoint family $(D_i)_{i\in I}$ of measurable sets of finite positive measure in the whole space $(Y,\Sigma,\nu)$, where $J\subseteq I$. (The additional sets are chosen among all measurable sets of finite positive measure.) Applying the first part of the lemma to the closed ideal $L^p(\nu)$ itself, we obtain a lattice isometry $L^p(\nu)\cong
	\left(\bigoplus_{i\in I}L^p(D_i,\nu)\right)_p$. Under this identification, $\mathcal J$ corresponds precisely to the coordinates indexed by $J$. Indeed, if a function $f\in\mathcal J$ were nonzero on some $D_i$ with $i\in I\setminus J$, then a positive measure level set of $|f|$ contained in $D_i$ would have its indicator in $\mathcal J$, contradicting the maximality of the original family. Conversely, surjectivity established above shows that every element supported on the coordinates indexed by $J$ belongs to $\mathcal J$. Thus, its disjoint complement $\mathcal J^\perp$ corresponds precisely to the remaining coordinates.
	The coordinate projections therefore yield the projection band $L^p(\nu)=\mathcal J\oplus_p\mathcal J^\perp$, as required.
\end{proof}

Now we are ready to prove Theorem~\ref{classification}.

\begin{proof}[Proof of Theorem~\ref{classification}]
	Let $\mathcal J\ne\{0\}$ be a closed ideal invariant under every translation, and choose a level set of a nonzero member to obtain a Borel set $E$ with $0<\mu(E)<\infty$ and $\1_E\in\mathcal J$. 
	For each $n\in\N$, write $x_n$ for the $n$-th coordinate of $x\in X$.
	Since $\mu(E)<\infty$, for every $n\in \N$ we can choose $b_n\ge1$ such that
		$$
		\mu(E\cap\{x\in \R^\N\colon |x_n|>b_n\})<2^{-n-1}\mu(E).
		$$
	Then the set $F:=E\cap\bigtimes_{n=1}^\infty[-b_n,b_n]$ has positive measure and $\1_F\in \mathcal J$. For each $a\in\{0,1\}^\N$, the translation $h_a:=(3b_na_n)_n$ gives a set $F+h_a$ whose indicator also belongs to $\mathcal J$. Furthermore, $(F+h_a)_{a\in\{0,1\}^\N}$ is a family of pairwise disjoint sets since their bounding boxes are disjoint in every coordinate where the corresponding binary sequences differ, and it also has cardinality $\cc$. Now by Zorn's lemma we can extend $(F+h_a)_{a\in\{0,1\}^\N}$ to a maximal one as in the hypothesis of Lemma~\ref{bands}, and its cardinality is exactly $\cc$ as $|\B_\infty|=\cc$. By Lemma~\ref{bands}, $\mathcal J$ is lattice isometric to the $\ell^p$-sum of the $L^p$-spaces associated with the sets in the maximal family. As each of these sets has finite positive measure, the classification of finite-measure pieces shows that each summand is lattice isometric to $L^p[0,1]$. As the family has cardinality $\cc$, we conclude that $\mathcal J\cong\ell^p(\cc,L^p[0,1])$ as Banach lattices, which includes the asserted identification of $L^p(\mu)$.
	
	\medskip
	
	Finally, the kernel $B^p$ is a closed ideal. It contains the indicators of $Q_t:=[t,t+1)^\N$, for $0<t<1$, since every coordinate overlap with an integer cube has length at most $\max(t,1-t)<1$. Moreover, $\mu(Q_t)=1$, and for distinct $s,t\in(0,1)$, we have
		$$
		\mu(Q_s\cap Q_t)=\lim_{N\to\infty}(1-|s-t|)^N=0.
		$$
	Extending this family within $B^p$, applying Lemma~\ref{bands} and using the local classification give the required lattice isometry for the kernel as well.
\end{proof}

For the whole space $L^p(\mu)$, the chosen family satisfies \eqref{exhaustion} for every finite-measure Borel set, and all its pieces have separable measure algebras. It is therefore a $\cc$-separable envelope in the sense of \cite[Def.~4.4]{previous}, disproving \cite[Conj.~4.6]{previous} directly. The same decomposition will be used for the ideals arising from translation continuity in the next section.

\section{Continuity of translation orbits}\label{sec3}

To apply the preceding classification to $C_v^p$, we first show that continuity of a vector is inherited by every function dominated by it in modulus. The argument uses the regularity of a finite Borel measure constructed from the vector.

\begin{lemma}\label{ideal}
	For every $v\in \R^\N$, the space $C_v^p$ is a closed ideal invariant under all translations.
\end{lemma}

\begin{proof}
	Closedness follows from the isometry of translations, while invariance under every $T_h$ follows from their commutativity. If $f\in C_v^p$ and $\varphi$ is bounded and continuous on $\R^\N$, then
		$$
		\|T_{tv}(\varphi f)-\varphi f\|_p \le \|\varphi\|_\infty\|T_{tv}f-f\|_p+\|(T_{tv}\varphi-\varphi)f\|_p\longrightarrow 0.
		$$
	Indeed, the first term tends to zero by the definition of $C_v^p$, while the second does so by dominated convergence, since $x-tv\to x$ in the product topology.
	Thus, $\varphi f\in C_v^p$.
	
	To prove the ideal property, let $f\in C_v^p$ and $|g|\le |f|$. We can write $g=\psi f$ for some bounded Borel function $\psi$. Consider the finite Borel measure $d\rho:=|f|^p\,d\mu$ on $\R^\N$. Since $\R^\N$ is Polish and $\rho$ is a finite Borel measure, $\rho$ is regular. Consequently, bounded continuous functions are dense in $L^p(\rho)$. Hence, there exists a sequence $(\varphi_n)_n$ of such functions converging to $\psi$ in $L^p(\rho)$. Consequently, $\varphi_n f\to g$ in $L^p(\mu)$, and the closedness of $C_v^p$ yields $g\in C_v^p$, as required.
\end{proof}

	The ideal property gives the band decomposition through Lemma~\ref{bands}, while the distinction between the three cases depends on the rectangles used to construct the measure. We begin with their density and then estimate their overlaps with translates.

\begin{lemma}\label{density}
	For $1\le p<\infty$, the linear span of $\{\1_R\colon R\in\F\}$ is dense in $L^p(\mu)$.
\end{lemma}

\begin{proof}
	Since finite-measure simple functions are dense in $L^p(\mu)$, it suffices to approximate $\1_E$ for $E\in\B_\infty$ with $\mu(E)<\infty$. Given $\varepsilon>0$, the definition in \eqref{outer} provides rectangles $R_k\in\F$ such that
		$$
		E\subset U:=\bigcup_{k=1}^\infty R_k,\qquad\sum_{k=1}^\infty\mu(R_k)<\mu(E)+\varepsilon.
		$$
	This covering satisfies $\mu(U\setminus E)<\varepsilon$, and the finiteness of $\mu(U)$ allows us to choose $N\in \N$ so that $U_N:=\bigcup_{k=1}^N R_k$ also satisfies $\mu(U\setminus U_N)<\varepsilon$. Consequently,
		$$
		\|\1_E-\1_{U_N}\|_p^p=\mu(E\triangle U_N)<2\varepsilon.
		$$
	To express this approximation in terms of rectangle indicators, we check that $\F$ is closed under finite intersections. For $R=\bigtimes_{n=1}^\infty A_n\in \F$ and $S=\bigtimes_{n=1}^\infty B_n\in \F$, the case $\vol(R)=0$ follows by comparison of finite products, which gives $\prod_{n=1}^\infty\lambda(A_n\cap B_n)=0$. When $\vol(R)>0$, all the factors $\lambda(A_n)$ are positive and
		$$
		\prod_{n=1}^N\lambda(A_n\cap B_n)=\left(\prod_{n=1}^N\lambda(A_n)\right)\prod_{n=1}^N\frac{\lambda(A_n\cap B_n)}{\lambda(A_n)}.
		$$
	The second product decreases to a limit in $[0,1]$, so $R\cap S\in \F$ in this case as well. Inclusion--exclusion now expresses $\1_{U_N}$ as a finite linear combination of rectangle indicators, and the preceding approximation proves the claimed density.
\end{proof}

The overlap estimate below will be used to exclude all nonzero continuous vectors in directions outside $\ell^1$. The assertion holds for almost every translation parameter with respect to Lebesgue measure on $\R$.

\begin{lemma}\label{overlap}
If $v\in \R^\N\setminus \ell^1$ and $R\in\F$ has positive measure, then $\mu(R\cap(R+tv))=0$ for almost every $t\in\R$.
\end{lemma}

\begin{proof}
	The proof begins with a one-dimensional estimate for measurable sets $A,E\subset\R$ with $a:=\lambda(A)\in(0,\infty)$ and $m:=\lambda(E)\in(0,\infty)$. For $b\ne0$, the Riesz rearrangement inequality in the form of \cite[Th.~1.2]{brascamp} applied to $\1_A(x)\1_A(x-bt)\1_E(t)$, gives
	\[
	\int_E\lambda(A\cap(A+bt))\,dt\le\int_{-m/2}^{m/2}(a-|b||t|)_+\,dt,
	\]
	where $r_+:=\max(r,0)$. Evaluating the right-hand side of the latter inequality and dividing by $a$ we obtain
		\begin{equation}\label{deficit}
		\int_E\left(1-\frac{\lambda(A\cap(A+bt))}{a}\right)dt\ge
		\begin{cases}
		|b|m^2/(4a) & \text{if } |b|m\le2a,\\
		m-a/|b| & \text{if } |b|m>2a.
		\end{cases}
		\end{equation}
	In particular, if $a\le2$, this lower bound is at least $c_m\min(|b|,1)$, where $c_m:=\min(m^2/8,m/2)>0$. The same lower bound holds for $b=0$ since both sides vanish.
	
	For the given rectangle $R=\bigtimes_{n=1}^\infty A_n$ and the sequence $v=(v_n)_n$, write its normalized coordinate overlaps as
		$$
		r_n(t):=\frac{\lambda(A_n\cap(A_n+tv_n))}{\lambda(A_n)}.
		$$
	The positive convergence of $\prod_{n=1}^\infty \lambda(A_n)$ implies $\lambda(A_n)\to1$, while translation continuity in $L^1(\R,\lambda)$ shows that each $r_n$ is continuous. Since these functions take values in $[0,1]$, the normalized finite products decrease and Baker's volume formula gives
		$$
		\mu(R\cap(R+tv))=\mu(R)\prod_{n=1}^\infty r_n(t).
		$$
	The convergence of the normalized products also shows that $R\cap(R+tv)\in\F$. At every parameter for which the overlap is positive, the logarithmic estimate
		$$
		\sum_{n=1}^\infty(1-r_n(t))\le-\log\prod_{n=1}^\infty r_n(t)<\infty
		$$
	shows that the sum of the coordinate deficits is finite.
	If the overlap $R\cap (R+tv)$ had positive $\mu$-measure for a set of parameters $t\in \R$ of positive Lebesgue measure, we could restrict to a bounded interval and then to a sublevel set of the sum $\sum_{n=1}^\infty(1-r_n(t))$ to obtain a Lebesgue measurable set $E$ and a constant $M>0$ with $0<\lambda(E)<\infty$ and $\sum_{n=1}^\infty(1-r_n(t))\le M$ for all $t\in E$. Since $\lambda(A_n)\to1$, there is $n_0\in \N$ such that $\lambda(A_n)\le2$ for all $n\ge n_0$, and Tonelli's theorem together with \eqref{deficit} yields
		$$
		M\lambda(E)\ge\sum_{n=n_0}^\infty\int_E(1-r_n(t))\,dt\ge c_m\sum_{n=n_0}^\infty\min(|v_n|,1)=\infty.
		$$
	The divergence of the last series follows from $v\notin\ell^1$ and contradicts the finiteness of $M\lambda(E)$, proving that the overlap vanishes almost everywhere.
\end{proof}

Now we are ready to prove Theorem~\ref{trichotomy}.

\begin{proof}[Proof of Theorem~\ref{trichotomy}]
	Assume first that $\supp v:=\{ n\in \N \colon v_n\neq 0 \}$ is finite (that is, $v\in c_{00}$) and consider a rectangle $R:=\bigtimes_{n=1}^\infty A_n\in\F$ with $\mu(R)>0$. Every factor $\lambda(A_n)$ is then positive, and changing only the coordinates in $\supp v$ gives
		$$
		\mu(R\cap(R+tv))=\mu(R)\prod_{n\in \supp v}\frac{\lambda(A_n\cap(A_n+tv_n))}{\lambda(A_n)}.
		$$
	Translation continuity in $L^1(\R,\lambda)$ implies that each factor on the right-hand side of the latter equality tends to $1$ as $t\to0$. Therefore
		\begin{equation}\label{finite}
		\|T_{tv}\1_R-\1_R\|_p^p=2\bigl(\mu(R)-\mu(R\cap(R+tv))\bigr)\longrightarrow0.
		\end{equation}
	Null rectangles contribute the zero vector, so Lemma~\ref{density} and the isometry of translations give continuity for every $f\in L^p(\mu)$.

	Consider the cube $C:=[0,1)^\N$. For every $v\in\R^\N$ and $t\in\R$, we have
		\begin{align}
			\|T_{tv}\1_C-\1_C\|_p^p & =\mu((C+tv)\triangle C) = \mu(C+tv)+\mu(C)-2\mu(C\cap(C+tv)) \nonumber \\
			& = 2\left(1-\mu\left(\bigtimes_{n=1}^\infty [0,1)\cap[tv_n,1+tv_n)\right)\right) \nonumber \\
			& = 2\left(1-\prod_{n=1}^\infty (1-|tv_n|)_+\right). \label{cube}
		\end{align}
	On the one hand, if $v\in\ell^1$, applying the finite-product inequality $1-\prod_n(1-u_n)\le\sum_n u_n$ to $u_n=\min\{|tv_n|,1\}$ and passing to the limit, we obtain $\|T_{tv}\1_C-\1_C\|_p^p \le 2|t|\|v\|_1\to 0$ as $t\to 0$. Thus, $\1_C$ is a nonzero continuous vector. On the other hand, if $v\notin\ell^1$ and $t\neq0$, the product in \eqref{cube} vanishes. Indeed, either one of its factors is zero, or all factors are positive and
		\[
		\prod_{n=1}^\infty(1-|tv_n|) \le \exp\left(-|t|\sum_{n=1}^\infty|v_n|\right) = 0.
		\]
	Consequently, $\|T_{tv}\1_C-\1_C\|_p^p=2$ for every $t\neq0$, so the orbit of $\1_C$ is discontinuous in this case.

	To prove that $C_v^p$ is proper whenever $\supp v$ is infinite, choose distinct indices $n_k$ with $v_{n_k}\ne0$ and positive integers $N_k$ satisfying $N_k|v_{n_k}|\ge k$. With these choices, define
		$$
		t_k := \frac1{N_k|v_{n_k}|},\qquad A_{n_k} := \bigcup_{j=0}^{N_k-1}\left[\frac{2j}{N_k},\frac{2j+1}{N_k}\right),
		$$
	and set $A_n:=[0,1)$ on all remaining coordinates. Note that every $A_n$ is a Borel subset of $[0,2]$ of measure one, so $R:=\bigtimes_{n=1}^\infty A_n \in \F$ and $\mu(R)=1$. Moreover, $t_k\to0$ and
		$$
		A_{n_k}\cap\left(A_{n_k}+\frac1{N_k}\right)=A_{n_k}\cap\left(A_{n_k}-\frac1{N_k}\right)=\varnothing.
		$$
	Since $t_kv_{n_k}=\pm1/N_k$, the sets $R$ and $R+t_kv$ are disjoint and their indicators have distance $2^{1/p}$ for every $k\in \N$. Thus the orbit of $\1_R$ is discontinuous, proving that $C_v^p$ is proper if $v\notin c_{00}$.
	
	To complete the trichotomy, we prove that $C_v^p=\{0\}$ when $v\notin\ell^1$. 
	
	Suppose that $v\notin\ell^1$ and $0\ne f\in C_v^p$. Choose $\varepsilon>0$ such that $E:=\{|f|>\varepsilon\}$ has finite positive measure. The inequality $\varepsilon\1_E\le|f|$ and Lemma~\ref{ideal} imply that $\1_E\in C_v^p$.
	By \eqref{outer}, there is a countable cover $E\subset\bigcup_{k=1}^\infty R_k$ with $R_k\in\F$ and $\sum_k\mu(R_k)<\infty$. Since $\mu(E)>0$, some rectangle $R$ in this cover satisfies $\mu(E\cap R)>0$. Set $F:=E\cap R$. Then $0<\mu(F)<\infty$, and the ideal property gives $\1_F\in C_v^p$. On the other hand, Lemma~\ref{overlap} yields
		$$
		\mu(F\cap(F+tv))\le\mu(R\cap(R+tv))=0
		\quad\text{for almost every }t\in\R.
		$$
	Consequently, for almost every $t\in\R$,
		$$
		\|T_{tv}\1_F-\1_F\|_p^p
		=2\bigl(\mu(F)-\mu(F\cap(F+tv))\bigr)
		=2\mu(F)>0.
		$$
	Every neighbourhood of zero contains parameters for which this equality holds, contradicting $\1_F\in C_v^p$. Thus $C_v^p=\{0\}$ when $v\notin\ell^1$, and the three alternatives follow for every $1\le p<\infty$.
	
	\medskip
	
	It remains to identify the bands. Lemmas~\ref{bands} and \ref{ideal} give $L^p(\mu)=C_v^p\oplus_p D_v^p$, with both summands invariant under all translations since translations preserve disjointness. Uniqueness of the decomposition shows that its projections commute with every translation, and when $v\in\ell^1\setminus c_{00}$ both summands are nonzero, so Theorem~\ref{classification} identifies $C_v^p$ and $D_v^p$ with $\ell^p(\cc,L^p[0,1])$.
	
	For $v\in\ell^1$, formula \eqref{cube} puts every integer-cube indicator in $C_v^p$, and the ideal property together with closure gives $G^p\subset C_v^p$. Hence $D_v^p\subset(G^p)^\perp=B^p$. Since each band component is dominated in modulus by the original function, band projections preserve every ideal, and their restrictions to $B^p$ yield $B^p=(B^p\cap C_v^p)\oplus_p D_v^p$.
	Finally, the indicators of $Q_t:=[t,t+1)^\N$, for $0<t<1$, belong to $B^p\cap C_v^p$ by translation invariance of the continuous band and the cube estimate. These sets form an essentially disjoint family of cardinality $\cc$, as shown in the proof of Theorem~\ref{classification}. Extend this family to a maximal essentially disjoint family of finite positive measure Borel sets whose indicators belong to the closed ideal $B^p\cap C_v^p$. Its cardinality is exactly $\cc$, since $|\B_\infty|=\cc$. Lemma~\ref{bands} and the lattice-isometric classification of the finite-measure pieces therefore give $B^p\cap C_v^p\cong\ell^p(\cc,L^p[0,1])$, completing the identification of the bands.
\end{proof}

The continuous part of the cube-restriction kernel contains functions whose restrictions to every translated unit cube vanish almost everywhere. The following example gives such a function and exhibits the proper inclusion of the cube-generated space in the continuous band.

\begin{example}\label{beyond}
Let $C:=[0,1)^\N$ and consider the closed translation-invariant subspace
	$$
	U^p:=\overline{\operatorname{span}}\{\1_{h+C}\colon h\in X\}.
	$$
Commutativity of translations and \eqref{cube} give $U^p\subset C_v^p$ whenever $v\in\ell^1$, but this inclusion is strict, as can be seen from the rectangle $R=A^\N$ with $A=[0,1/2)\cup[3/2,2)$.
Every interval of length one intersects $A$ in a set of measure at most $1/2$, so $\mu(R)=1$ while $\mu(R\cap(h+C))=0$ for every $h\in X$. Restriction to $R$ therefore annihilates $U^p$, whereas $\1_R\in B^p$ has norm one, and consequently $\dist(\1_R,U^p)=1$.

On the other hand, $\lambda(A\cap(A+s))=1-2|s|$ for $|s|\le1/2$. For $v\in\ell^1$ and $|t|\|v\|_\infty\le1/2$, the same finite-product estimate as above gives
	$$
	\|T_{tv}\1_R-\1_R\|_p^p=2\left(1-\prod_{n=1}^\infty(1-2|tv_n|)\right)\le4|t|\|v\|_1.
	$$
Hence, $\1_R \in C_v^p$ for every $v\in\ell^1$, and $U^p\subsetneq C_v^p\subsetneq L^p(\mu)$ for every $v\in\ell^1\setminus c_{00}$. The same rectangle separates $\1_R$ from the closed span of translated cubes with finitely many measurable coordinate modifications used in \cite[Lem.~5]{sakbaev} since each such set is contained in a translated unit cube. Moreover, $R\subset\ell^\infty$, so this phenomenon occurs among bounded sequences and depends on the choice of measure.
\end{example}

Thus the cube-restriction kernel contains continuous vectors that vanish on every translated unit cube. Theorem~\ref{trichotomy} identifies the isometric types of the resulting bands, while a direct description of membership in the continuous band in terms of the coordinate geometry of measurable sets remains a further question.

\subsection*{Declaration of AI Use}

The authors used ChatGPT (OpenAI, GPT-5.6 and GPT-6) during the development and preparation of this manuscript for mathematical exploration, testing of arguments and assistance with exposition. All mathematical claims and proofs were reviewed by the authors, who assume full responsibility for the content of the work.

\end{document}